\def\version{22 January 2016}
 
\documentclass[reqno]{amsart}
\usepackage{amsmath}
\usepackage{amsfonts}
\usepackage{amssymb}
\usepackage{amssymb,graphics,graphicx,bbm,hyperref,color}

\usepackage{multicol}
\usepackage{enumitem} 
\usepackage{dsfont}

\definecolor{gray}{rgb}{0.93,0.93,0.93}
\definecolor{light-gold}{rgb}{0.99,0.97,0.78}

\def\be{\begin{equation}}
\def\ee{\end{equation}}
\def\bm{\begin{multline}}
\def\bfig{\begin{figure}[htb]}
\def\efig{\end{figure}}

\newcommand{\abs}[1]{{\lvert #1\rvert}}

\numberwithin{equation}{section}
\newtheorem{theorem}{Theorem}[section]

\newtheorem{lemma}[theorem]{Lemma}
\newtheorem{corollary}[theorem]{Corollary}

\theoremstyle{remark}

\newcommand{\caC}{{\mathcal C}}

\newcommand{\bbE}{{\mathbb E}}
\newcommand{\bbN}{{\mathbb N}}
\newcommand{\bbP}{{\mathbb P}}

\newcommand{\bbZ}{{\mathbb Z}}
\newcommand{\bse}{{\boldsymbol e}}

\newcommand{\eps}{{\varepsilon}}

\newcommand{\co}{{\rm c}}
\newcommand{\1}{{\mathds1}}

\makeatletter
  \def\tagform@#1{\maketag@@@{\footnotesize{(#1)}\@@italiccorr}}
\makeatother

\renewcommand{\eqref}[1]{(\ref{#1})}

\theoremstyle{plain}

\global\long\def\cbr#1{\left\{  #1\right\}  }
\global\long\def\rbr#1{\left(#1\right)}

\begin{document}

\phantom{{\hfill\small \version} \vspace{2mm}}

\title{The random interchange process on the hypercube}

\author{Roman Koteck\'y}
\address{Department of Mathematics, University of Warwick,
Coventry, CV4 7AL, United Kingdom, and Centre for Theoretical Study, Charles University, Prague, Czech Republic}
\email{R.Kotecky@warwick.ac.uk}

\author{Piotr Mi\l{}o\'{s}}
\address{Faculty of Mathematics, Informatics and Mechanics, University of Warsaw, Banacha 2,
02-097 Warszawa, Poland}
\email{pmilos@mimuw.edu.pl}

\author{Daniel Ueltschi}
\address{Department of Mathematics, University of Warwick,
Coventry, CV4 7AL, United Kingdom}
\email{daniel@ueltschi.org}



\begin{abstract}
We prove the occurrence of a phase transition accompanied by the emergence of
cycles of diverging lengths in the random interchange process on the hypercube. 
\end{abstract}

\thanks{\copyright{} 2016 by the authors. This paper may be reproduced, in its entirety, for non-commercial purposes.}

\maketitle

\section{Introduction}
\label{sec intro}

The interchange process is defined on a finite graph. With any edge is associated the transposition of its endvertices. The outcomes of the interchange process consist of sequences of random transpositions and the main questions of interest deal with the cycle structure of the random permutation that is obtained as the composition of these transpositions. As the number of random transpositions increases, a phase transition may occur that is 
indicated by the emergence of cycles of diverging lengths involving a positive density of vertices.

The most relevant graphs are regular graphs with an underlying ``geometric structure'' like
a finite cubic box in $\bbZ^{d}$ with edges between nearest neighbours. 
But the problem of proving the emergence of long cycles is out of reach for now and recent studies have been devoted to simpler graphs such as trees \cite{Ang,Ham} and complete graphs \cite{Sch,Ber,BK}. (Note also the intriguing identities of Alon and Kozma based on the group structure of permutations \cite{AK}.) The motivation for the present article is to move away from the complete graph towards $\bbZ^{d}$. 
 We consider the hypercube $\{0,1\}^{n}$ in the large $n$ limit and  establish the occurrence of a phase transition demonstrated by the emergence of cycles larger than $2^{(\frac12-\varepsilon) n}$. Our proof combines the recent method of Berestycki \cite{Ber}, which was used for the complete graph but is valid more generally, with an estimate of the rate of splittings that involves the isoperimetric inequality for hypercubes.

Besides its interest in probability theory, the random interchange process appears in studies of quantum spin systems \cite{Toth}, see also \cite{GUW} for a review. Cycle lengths and cycle correlations give information on the magnetic properties of the spin systems. The setting is a bit different, though. First, the number of transpositions is not a fixed parameter, but a Poisson random variable. Second, there is an additional weight of the form $\theta^{\# \text{cycles}}$ with $\theta=2$ in the case of the spin $\frac12$ quantum Heisenberg model. The first feature is not a serious obstacle, but the second feature turns out to be delicate. Notice that Bj\"ornberg recently obtained results about the occurrence of macroscopic cycles on the complete graph in the case $\theta>1$ \cite{Bjo,Bjo2}.
Correspondence between random transpositions and quantum models when $\theta=3,4,...$, and with a more general class of random loop models, can be found in \cite{AN,Uel}. We expect that our hypercube results can be extended to these situations as well, but it may turn out not to be entirely straightforward.

\section{Setting and results}

Let $G_n=(Q_{n},E_{n})$ be  a graph whose $N = 2^{n}$ vertices form a \emph{hypercube}  $Q_{n} = \{0,1\}^{n}$ with edges joining nearest-neighbours---pairs
of vertices that differ in exactly one coordinate,
$E_n=\{\{x,y\}: x,y\in Q_{n}, \abs{x-y}_1=1 \}$, $\abs{E_{n}}=\frac{Nn}2$.

Let $\Omega_n$ be the set of infinite sequences of edges in $E_n$. For $t\in \mathbb{N}$ by $\mathcal{F}_{n,t}$ we denote the $\sigma$-algebra generated by the first $t$ elements of the sequence. Further, for $t\in \mathbb{N}$ {we use $\Omega_{n,t}$ to denote} the set of sequences of $t$ edges $\bse = (e_{1},\dots,e_{t})$, where $e_{s} \in E_{n}$ for all $s=1,\dots,t$. The $\sigma$-algebra $\mathcal{F}_{n,t}$ will be identified with the total $\sigma$-algebra over $\Omega_{n,t}$. For an event $A\in \mathcal{F}_{n,t}$ we set
\[
	\bbP_{n}(A) = |A|\bigl(\tfrac2{Nn}\bigr)^t,
\]
i.e. edges are chosen independently and uniformly from $E_n$. 

Using $\tau_{e}$ to denote the transposition of the two endvertices of an edge $e \in E_{n}$,  we can view the sequence $\bse \in \Omega_{n,t}$ as a series of \emph{random interchanges}  generating a \emph{random permutation} $\sigma_t = \tau_{e_{t}}\circ  \tau_{e_{t-1}} \circ \cdots \circ \tau_{e_{1}}$ on $Q_{n}$. For any $\ell \in \bbN$, let $V_{t}(\ell)$ be the random set of vertices that belong to permutation cycles of lengths greater than $\ell$ in  $\sigma_t$. 

We start with the straightforward observation that only small cycles occur in $\sigma_t$ when $t$ is small.
It is based on the fact that the random interchange model possesses a natural percolation structure when viewing any edge contained in $\bse$ as opened.
The probability that a particular edge remains closed by the time $t$ is $\bigl(1-\frac{1}{Nn/2}\bigr)^t$.
Since the set of vertices of any cycle must be contained in a single percolation cluster, only small cycles occur when percolation clusters are small.

\begin{theorem}
\label{T:subcrit}
Let $c<1/2$ and $\epsilon>0$. Then there exists $n_0$ such that 
\[
\bbP_{n}(|V_{t}(\kappa n)| = 0) > 1 - \epsilon  \kappa^{-3/2}
\]
for all $t \le c N$, all $\kappa\ge \frac{2\ln 2}{(1-2c)^2}$, and all $n>n_{0}$.
\end{theorem}

\begin{proof}
In view of the above mentioned percolation interpretation of the random interchange model, the claim follows from the fact that  
the percolation model on the hypercube graph $Q_n$ is subcritical for $p=2c/n$ with $c<1/2$ and  the size of the largest cluster is of the order $n$ (see \cite{AKS}). The value $p=2c/n$ corresponds to  $t= c N$ implying that the probability of any particular edge to be open is
$1-\bigl(1-\frac{2}{Nn}\bigr)^{c N}\sim \frac{2c}n$. The claim of the theorem follows from \cite[Theorem 9]{BKL}. In particular, the last displayed inequality in its proof can be reinterpreted  as a claim that
\be
\label{E:EVlarge}
\bbE_{n}(|V_{t}(\kappa n)|) \leq \epsilon(n) \kappa^{-3/2}
\ee
with $\epsilon(n)\to 0$ as $n\to\infty$  whenever  $\kappa>\frac{2\ln 2}{(1-2c)^2}$.
\end{proof}

Our main result addresses the emergence of long cycles for 
large times, $t>N/2$. We expect that cycles of order $N$ occur for all large times; here we prove 
a weaker claim:  cycles larger than $N^{\frac12-\varepsilon}$ occur for a ``majority of large times''.

\begin{theorem}
\label{T:supercrit}
Let $c>\frac12$ and let $(\Delta_{n})$ be a sequence of positive numbers such that $\Delta_{n} n/\log n \to \infty$ as $n \to \infty$. Then there exist $\eta(c)>0$ and $n_{0}$ such that for all $n>n_{0}$, all $T> c N$, and all $a>0$, we have
\[
 \frac1{\Delta_{n} T} \sum_{t =T+1}^{\lfloor (1+\Delta_{n}) T \rfloor}    \bbE_{n} \Bigl( \frac{|V_{t}(N^{a})|}{N} \Bigr) \ge \eta(c) - a.
\]
For $c>1$, we can take $\eta(c) = \tfrac12 (1 - \frac1c)$.
\end{theorem}

Let us observe that the highest achievable value of the exponent $a$ is just below $1/2$; this can be accomplished only with $c$ becoming large. But we expect that the size of the long cycles is of order $N$. In fact, one can formulate a precise conjecture, namely that the joint distribution of the lengths of long cycles is Poisson-Dirichlet. This was proved in the complete graph \cite{Sch}, and advocated in $\bbZ^{d}$ with $d\geq3$ \cite{GUW}.

The proof of Theorem \ref{T:supercrit} can be found in Section \ref{sec proof thm}; it is based on a series of lemmas obtained in the next section.

We can choose $\Delta_{n} \equiv \Delta > 0$, rather than a sequence that tends to 0. In this case, Theorem \ref{T:supercrit} takes a simpler form, which perhaps expresses the statement `long cycles are likely' more directly.

\begin{corollary}
	Let $a\in(0,1/2)$, $\Delta>0$, and $\epsilon_1 \in (0,\frac12-a)$. Then there exists $c> 1$ and $\epsilon_2 >0$ such that for $n$ large enough we have
	\[
		\frac{1}{\Delta T}  \sum_{t =T+1}^{\lfloor (1+\Delta) T \rfloor}  \bbP_{n}\Bigl(\frac{|V_t(N^a)|}{N}\geq \epsilon_1 \Bigr) \geq  \epsilon_2
	\]
	for all $T> c N$.
\end{corollary}

\begin{proof}
This follows from Theorem \ref{T:supercrit} and Markov's inequality. Namely,
\be
\begin{split}
\bbP_{n}\Bigl(\frac{|V_t(N^a)|}{N}\geq \epsilon_1 \Bigr) &= 1 - \bbP_{n}\Bigl(1-\frac{|V_t(N^a)|}{N}\geq 1-\epsilon_1 \Bigr) \\
&\geq 1 - \frac1{1-\epsilon_{1}} \bbE_{n}\Bigl(1-\frac{|V_t(N^a)|}{N}\Bigr) \\
&\geq 1 - \frac{1 - \eta(c)+ a}{1-\epsilon_{1}}.
\end{split}
\ee
This is positive for $\epsilon_{1} < \eta(c)-a$.
\end{proof}

\section{Occurrence of long cycles}
\label{sec proofs}

\subsection{Number of cycles vs number of clusters}
\label{sec Ber}

Cycle structure and percolation properties are intimately related, and we will rely on Berestycki's key observation that the number of cycles remains close to the number of clusters \cite{Ber}.
Let $N_{t}$ denote the random variable for the number of cycles of the random permutation $\sigma_t$ at time $t$, and $\widetilde N_{t}$ the number of clusters
of the underlying percolation model. Notice that $N_{t} \geq \widetilde N_{t}$. 

Let us consider the possible outcomes when a new random transposition arrives at time $t$. 
 There are three possibilities;
the endpoints of a new edge $e_t$ are either both in the same cycle of $\sigma_{t-1}$ (and thus also in the same cluster), or in the same cluster but in different cycles, or in different clusters.
Correspondingly, we are distinguishing three events:
\begin{itemize}
\item $S_{t}$,  a splitting of a cycle  where $N_{t} = N_{t-1}+1$ and $\widetilde N_{t} = \widetilde N_{t-1}$. Indeed, a splitting of any cycle  necessarily occurs within the same percolation cluster.
\item $M_{t}$, a  merging of two cycles within the same cluster: $N_{t} = N_{t-1}-1$ and $\widetilde N_{t} = \widetilde N_{t-1}$.
\item $\widetilde M_{t}$, a merging of two cycles in distinct clusters: $N_{t} = N_{t-1}-1$ and $\widetilde N_{t} = \widetilde N_{t-1}-1$.
\end{itemize}
Also, let $I_t=S_t\cup M_t$ be the event where the endpoints of  the edge $e_t$ belong to the same cluster. Notice that $\widetilde M_{t} = I_{t}^{\rm c}$.
Obviously, the three events above are mutually disjoint and cover all outcomes,
\be
\label{E:SMM}
\Omega_{n,t} =S_{t} \cup M_{t} \cup \widetilde M_{t}.
\ee
Notice that 
\be
\label{E:N-N}
N_{t} - \widetilde N_{t} = \sum_{i=1}^{t} ({\mathds1}_{S_{i}} - {\mathds1}_{M_{i}}).
\ee
A key in the proof of Theorem~\ref{T:supercrit} is the isoperimetric inequality of the hypercube $Q_n$. Namely, for any set $A\subset Q_n$, the number $\abs{E(A)}$ of edges of $G_n$ whose both end-vertices are in $A$ is
 \be
 \label{E:isoperimetric}
 \abs{E(A)}\le \tfrac12\abs{A}\log\abs{A}.
 \ee
Here (and elsewere in this paper) $\log$ is always meant as the logarithm of  base $2$.
See \cite{BL} for the proof of the bound in this form. 
It implies a lower bound on the number $\abs{E(A|A^{\co})}$ of edges connecting $A$ with its complement $A^{\co}=Q_n\setminus A$, namely
 \be
\abs{E(A|A^{\co})}\ge \abs{A}(n-\log\abs{A}).
\ee
We are not referring to this inequality in this article, but we found it useful in discussions.

Theorem~\ref{T:supercrit} would follow from the following lemma once its assumption is proven.

\begin{lemma}
\label{L:main}
Assume that $\bbP_{n}({S_{t}}) \geq \lambda$ with $\lambda\in(0,1)$. Then
\[
\bbE_{n} \Bigl( \frac{\abs{V_{t}(N^{a})}}N \Bigr) \geq \frac{{\lambda} - a}{1-a}
\]
{ for any $a\in(0,\lambda)$.}
\end{lemma}

\begin{proof}
Let $\caC_{t-1}$ denote the set of cycles at time $t-1$. Since the total number of edges is $Nn/2$, and at most $\frac12 \sum_{C \in \caC_{t-1}} |C|  \log |C|$ edges cause a splitting, we have
\be
\bbP_{n}(S_{t}|\caC_{t-1}) \leq \frac1{Nn} \sum_{C \in \caC_{t-1}} |C| \log |C|.
\ee
It follows that 
\be
\begin{split}
\lambda &\leq \bbP_{n}(S_{t}) = \bbE_{n} ( \bbP_{n}(S_{t}|\caC_{t-1})) \\
&\leq \frac1{Nn} \bbE_{n} \Bigl( \sum_{C\in\caC_{t-1}} |C| \log |C| \Bigr) \\
&= \frac1{Nn} \bbE_{n} \Bigl( \sum_{C\in\caC_{t-1} : |C| \leq N^{a}} |C| \log |C| \Bigr) + \frac1{Nn} \bbE_{n} \Bigl( \sum_{C\in\caC_{t-1} : |C| > N^{a}} |C| \log |C| \Bigr) \\
&\leq \frac aN \bbE_{n} \Bigl( \sum_{C\in\caC_{t-1} : |C| \leq N^{a}} |C| \Bigr) + \frac1N \bbE_{n} \Bigl( \sum_{C\in\caC_{t-1} : |C| > N^{a}} |C| \Bigr).
\end{split}
\end{equation}
Using  $\sum_{C\in\caC_{t-1}} |C|=N$ and $\sum_{C\in\caC_{t-1}: |C| > N^{a}} |C| = |V_{t}(N^{a})|$, we get the lemma.
\end{proof}

What remains to be done is to establish {a} lower bound on the probability for an edge to connect {vertices within a  cycle and thus splitting it. 
We will get it by combining  lower bounds on the probability $\bbP_{n}(I_{t})$   for an edge to connect vertices within one cluster and on the rate $\bbP_{n}(S_{t})/\bbP_{n}(I_{t})$ for  those actually connecting vertices within a  cycle}.

As it turns out,  we can verify the {latter lower bound} only in a mean sense, averaging over an interval $[{T}, T+L]$ where ${T}$ is large. The ratio $L/T$ can be chosen to vanish but not too fast. We will use the following corollary, whose proof is  essentially a verbatim repetition of the proof above.

\begin{corollary}
\label{C:averaging}
Assume that for some $T,L\in\mathbb N$,  and $\lambda\in(0,1)$, we have
\[
\frac1L \sum_{t = T+1}^{T+L} \bbP_{n}(S_{t}) \geq  \lambda.
 \]
Then
\[
\frac1L \sum_{t = T+1}^{T+L} \bbE_{n} \Bigl( \frac{|V_{t}(N^{a})|}N \Bigr) \geq 
\frac{\lambda - a}{1-a}
\]
 for any $a\in(0,\lambda)$.
\end{corollary}

\subsection{Lower bound on the probability of $I_t$}
\label{S:It}

Here we show that, if the time is large enough, there is a positive probability that the vertices of a random edge belong to the same cluster.  
Equivalently, we need an upper bound on the probability  of the event ${\widetilde M}_{t}=I_t^{\co}$ that two clusters are merging. The first lemma applies to $c>\frac12$; the second lemma is restricted to $c>1$ but it gives an explicit bound. Let $\tilde V_{t}$ denote the largest percolation cluster after $t$ random transpositions.

\begin{lemma}
\label{lem AKS}
Assume that $\bbE_{n}(|\tilde V_{t}|) > c_{0}N$ for some constant $c_{0}>0$. Then there exists $c'>0$ such that
\[
\bbP_{n}(I_{t}) > c' c_{0} (1-o(1)).
\]
\end{lemma}

\begin{proof}
It is based on \cite[Remark 2]{AKS}, which states that there exist $\eps>0$ and $c'>0$ such that
\be
\label{W}
\bbP_{n}(|W_{t}| > N - N^{1-\eps}) = 1 - o(1),
\ee
where $W_{t}$ is the set of vertices which have at least $c'n$ neighbours in $\tilde V_{t}$. By only considering edges within the largest cluster, we obtain
\be
\bbP_{n}(I_{t}) \geq \frac{c'}N \bbE_{n}(|\tilde V_{t} \cap W_{t}|).
\ee
Using $|\tilde V_{t} \cap W_{t}| \geq |\tilde V_{t}| - |W_{t}^{\rm c}|$, the lemma follows.
\end{proof}

\noindent
{\it Remark:} In \cite{AKS} the authors use their Remark 2 as an indication that for $t>\frac{1}2 N$ the second largest cluster is of size $o(N)$. Actually, this has been proven in
\cite[Theorem 31]{BKL} where it was shown that the size of the second largest cluster is of the order at most $n/(2c-1)^2$ (we adhere here to our notation with critical $c=1/2$).  The claim \eqref{W} thus follows immediately, combining  \cite[Lemma 3]{AKS} --- which states the same for the set of all vertices that have at least $c'n$ neighbours in clusters of size at least $n^2$, with \cite[Theorem 31]{BKL} --- which implies that this set actually coincides with $W_t$.
(Notice that in both \cite{AKS} and \cite[Theorem 31]{BKL}, the results are actually formulated  for  percolation clusters on the hypercube  with probability of an edge being occupied chosen as
$p=2c/n$.)

We state and prove the next lemma for the hypercube, but it actually holds for any finite graph.

\begin{lemma}
\label{lem slow merge}
Let $t\in \mathbb{N}$ and $\delta\in (0,1)$. Then
\[
\bbP_{n}({\widetilde M}_{t}) \leq N/t+t^{-(1-\delta)/2}+\exp\rbr{-t^{\delta}/2}.
\]
\end{lemma}

\begin{proof}
We recall that $\cbr{\mathcal{F}_{t}}_{t\geq0}$ denotes the filtration associated with the process of adding edges (i.e. $\mathcal{F}_{t}$ contains
information about the first $t$ edges), and define the random variables
\be
p_{t} = \bbE_{n}({\1}_{{\widetilde M}_{t}}|\mathcal{F}_{t-1}).
\ee
Simple but crucial observations are that for  any  $t \in \bbN$, we have 
\begin{equation}
1\geq p_{t}\geq p_{t+1}\ \text{ and }\quad\sum_{i=1}^{t}\1_{{\widetilde M}_{i}}\leq N.
\label{eq:observations}
\end{equation}
Indeed, adding an edge decreases the chance of next merging and the
total number of mergings is smaller than the size of the graph. Let us define the process
$\cbr{X_{t}}_{t\geq0}$ by 
\be
X_{t} = \sum_{i=1}^{t}\left(\1_{{\widetilde M}_{i}}-p_{i}\right).
\ee
One verifies that it is a martingale and $|X_{t+1}-X_{t}|\leq1$.
By the Azuma inequality we have $\bbP_{n} \left(X_{t}\leq-t^{(1+\delta)/2}\right)\leq\exp\rbr{-t^{1+\delta}/(2t)}=\exp\rbr{-t^{\delta}/2}.$
Hence,
\be
\begin{split}
1-\exp\rbr{-t^{\delta}/2} &\leq \bbP_{n} \Bigl(\sum_{i=1}^{t}\bigl(\1_{{\widetilde M}_{i}}-p_{i}\bigr)\geq-t^{(1+\delta)/2}\Bigr) \\
&= \bbP_{n}\Bigl(\sum_{i=1}^{t}\1_{{\widetilde M}_{i}}\geq-t^{(1+\delta)/2} + \sum_{i=1}^{t}p_{i}\Bigr).
\end{split}
\ee
Using (\ref{eq:observations}) we estimate further,
\be
\begin{split}
1-\exp\rbr{-t^{\delta}/2} &\leq \bbP_{n} \Bigl(\sum_{i=1}^{t}\1_{{\widetilde M}_{i}}\geq-t^{(1+\delta)/2}+tp_{t}\Bigr) \\
&\leq \bbP_{n} \bigl(N+t^{(1+\delta)/2}\geq tp_{t}\bigr).
\end{split}
\ee
In other words $\bbP_{n}(p_{t}\geq N/t+t^{-(1-\delta)/2})\leq\exp\rbr{-t^{\delta}/2}.$
Finally, 
\be
\begin{split}
\bbP_{n}({\widetilde M}_{t}) &= \bbE_{n}(\1_{{\widetilde M}_{t}}) \\
&= \bbE_{n}(p_{t}) \\
&\leq \bbP_{n} \bigl( p_{t} \geq N/t+t^{-(1-\delta)/2} \bigr) + N/t+t^{-(1-\delta)/2} \\
&\leq \exp\bigl(-t^{{\delta}/2}\bigr)+N/t+t^{-(1-\delta)/2}.\qedhere
\end{split}
\ee
\end{proof}

\subsection{Lower bound on the rate $ \bbP_{n}(S_t)/\bbP_{n}(I_t)$}
\label{S:St/It}

Let us begin with a bound on the probability of unfavourable splittings that result in short cycles. We define  the event $S_{t}^{\leq k}\subset S_t$  as those splittings that result in  producing a cycle of length less than or equal to $k$ (or in two such cycles). 

\begin{lemma}
\label{L:splitprob}
For any $n,t,k \in \bbN$ we have  
\[
\bbP_{n} \bigl( S_{t}^{\leq k} \bigr) \leq \frac{2\log(2k)}{n}.
\]
\end{lemma}

\begin{proof}
Given an arbitrary configuration $\bse\in\Omega_{n,t}$ yielding a collection of cycles covering $Q_n$, we can find a family of sets  $\{A_{i}\}, A_i\subset Q_n$ such that
\begin{enumerate}
\item[(a)] $|A_{i}| \leq 2k$;
\item[(b)] $\sum_{i} |A_{i}| \leq 2N$;
\item[(c)] ${\mathds1}_{S_{t}^{\le k}} \leq \sum_{i} {\mathds1}_{e_{t} \in E(A_{i})}$, where $e_t \in E_{n}$ is the random edge at time $t$.
\end{enumerate}
Indeed, to each cycle of length less than or equal to $2k$, we define $A_{i}$ to be its support. For a cycle 
of length $\ell>2k$, we label its vertices consecutively by natural numbers (starting from an arbitrary one) identifying the labels $j, \ell+j, 2\ell+j,\dots$, $j=1,\dots, \ell$. Denoting $m= \lfloor \ell/k \rfloor$, notice that  $mk\le\ell<(m+1)k$ and $\ell+k<(m+2)k<2\ell$.
We cover the cycle
by the following collection of intervals
$$(1,\dots, 2k), (k+1,\dots, 3k), \dots, ((m-1)k+1,\dots,(m+1)k), (m k+1,\dots,\ell + k)$$
if $mk<\ell$. In the case $mk=\ell$, the last interval is skipped and the collection ends with $((m-1)k+1,\dots,(m+1)k)$.
 Clearly, the length of all intervals is either $2k$ or, for the last one, $\ell+k-mk< 2k$ thus (a) holds. Further, (b) is implied by the fact that any site of the cycle is covered exactly twice. 
 Moreover, any pair $j_1<j_2$ such that $j_1 \in(1,\dots \ell)$ and $j_2-j_1<k$ is necessarily contained in at least one of above intervals.
 Namely,  if $j_1\in (rk+1, \dots, (r+1)k)$ (resp. $j_1\in (m k+1,\dots,\ell)$ for the last interval if $mk<\ell$),
 then $j_1,j_2\in (rk+1, \dots, (r+2)k)$ (resp. $j_1,j_2\in (m k+1,\dots,\ell+k)$).
As a result, (c) is verified and thus
we get 
\be
\label{bound involving A_i}
\bbP_{n}(S_{t}^{\leq k}) \leq \sum_{i} \bbP_{n}(e_t\in E(A_{i})).
\ee
The number of edges  $E(A_{i})$ induced by $A_{i}$ is, according to the isoperimetric inequality \eqref{E:isoperimetric}, bounded by 
$\frac12 |A_{i}| \log |A_{i}|$. Given that the number of all edges in $Q_n$ is $\frac{Nn}2$, we get $\bbP_{n}(e_t\in E(A_{i}))\le \frac{|A_{i}| \log |A_{i}|}{Nn}\le |A_{i}|\frac{ \log (2k)}{Nn}$.  Using also that $\sum_{i}|A_{i}|\le 2N$, we  get the claimed bound
\be
\bbP_{n}(S_{t}^{\leq k}) \leq \sum_{i}|A_{i}|\frac{ \log (2k)}{Nn}\le 2N  \frac{ \log (2k)}{Nn}= \frac{2\log(2k)}{n}.
\ee
\end{proof}

The  bound from Lemma~\ref{L:splitprob} can be used to show that the number of cycles $N_t$ does not depart too far from the number of clusters
${\widetilde N}_t$.
\begin{lemma}
\label{L:N-Nbound}
    There exists $n_{1}$ such that for $n \geq n_{\color{blue}{1}}$ and any $t \in \bbN$ we have 
\[
\bbE_{n}(N_{t}-\widetilde N_{t}) \leq t\frac{3 \log (4n)}{n}.
\]
\end{lemma}

\begin{proof}
Let $N^{\le 2n}$ (resp. $N^{> 2n}$) denote the number of cycles shorter {or equal} to $2n$ (resp. longer than $2n$). Obviously  $N_{t} = N_{t}^{\le 2 n} + N_{t}^{> 2 n}$ and thus
\be
\begin{split}
\bbE_{n}(N_{t} - \widetilde N_{t}) &= \bbE_{n}(N_{t}^{\leq2n} - \widetilde N_{t}) + \bbE_{n}(N_{t}^{>2n}) \\
&\leq \bbE_{n} \Bigl( \sum_{i=1}^{t} {\mathds1}_{S_{i}^{\leq 2 n}} \Bigr) + \bbE_{n}(N_{t}^{>2n}).
\end{split}
\ee
To bound $\bbE_{n}(N_{t}^{>2n})$, we simply use that $N_{t}^{>2n} \leq N/(2n)\le t \frac{\log(2n)}{n}$ once $t> \frac{N}{2\log(2n)}$. 
On the other hand,  for $t\le \frac{N}{2\log(2n)}$ we get $\bbE_{n}(N_{t}^{>2n}) \leq 1/n$  once $n$ is sufficiently large.
Indeed, observe that $N_{t}^{> 2 n}\le \frac{|V_t(2n)|}{2n}$. Hence,  we can use \eqref{E:EVlarge} with $c=1/\log(2n)$ allowing to choose $\kappa=2 >\frac{2\ln 2}{(1-2c)^2}$. 
 
The result then follows from Lemma \ref{L:splitprob}.
\end{proof}

\begin{lemma}
\label{L:many splits}
For any $T,L \in \mathbb{N}$ and any $n \geq n_{1}$  (with $n_{1}$ the constant from Lemma \ref{L:N-Nbound}), we have
\[
\sum_{t= T+1}^{T+L} \bbP_{n}(S_{t}) \geq \tfrac12 \sum_{t= T+1}^{T+L} \bbP_n(I_{t}) - \tfrac32  T \frac{\log (4n)}{n}.
\]
\end{lemma}

\begin{proof}
We have
\be
\begin{split}
\sum_{t= T+1}^{T+L} {\mathds1}_{S_{t}} &= \tfrac12 \sum_{t= T+1}^{T+L} \bigl( {\mathds1}_{S_{t}} - {\mathds1}_{M_{t}} \bigr) + \tfrac12 \sum_{t= T+1}^{T+L}  \bigl( {\mathds1}_{S_{t}} +{\mathds1}_{M_{t}} \bigr) \\
&= \tfrac12 \bigl( N_{T+L} - \widetilde N_{T+L} - N_{{T}} + \widetilde N_{{T}} \bigr) + \tfrac12 \sum_{t= T+1}^{T+L}  {\mathds1}_{I_{t}}.
\end{split}
\ee
 The claim follows by taking expectations,  using that $N_{T+L}-\widetilde N_{T+L}\geq0$, and applying Lemma~\ref{L:N-Nbound} for the expectation of $N_{ T}-\widetilde N_{ T}$. 
\end{proof}

\subsection{Proof of Theorem \ref{T:supercrit}}
\label{sec proof thm}

We check the condition of Corollary~\ref{C:averaging} with $\lambda=\eta(c)$, where
\be
\eta(c) = \begin{cases} \frac12 c' c_{0} & \text{if } c \in (\frac12, 1], \\ \frac12 (1 - \frac1c) & \text{if } c>1. \end{cases}
\ee
By Lemma~\ref{L:many splits} and Lemmas \ref{lem AKS} with \ref{lem slow merge}, we have
\be
\frac1L \sum_{t=T+1}^{T+L} \bbP_{n}(S_{t}) \geq \tfrac12 \frac1L \sum_{t=T+1}^{T+L} \bbP_{n}(I_{t}) -\tfrac32  \frac{T}{L} \frac{\log (4n)}{n} \geq \eta(c) - o(1)
\ee
once we choose $L=\Delta_n T$ with $\Delta_{n} n/\log n \to \infty$. 
Theorem \ref{T:supercrit} now follows from Corollary~\ref{C:averaging}, since $\frac{\eta(c)-a}{1-a} > \eta(c)-a$; this actually allows to neglect the corrections $o(1)$.
\hfill$\square$

\medskip
\noindent
{\bf Acknowledgments:} We are grateful to Nathana\"el Berestycki for clarifying to us that his result \cite{Ber} applies to arbitrary graphs of diverging degrees. We also thank the referee for bringing our attention to Remark 2 in \cite{AKS}; this allowed us to extend Theorem \ref{T:supercrit} from $c>1$ to $c>\frac12$.
D.U.\ thanks the Newton Institute for a useful visit during the program ``Random Geometry'' in the spring 2015.  The research of R.K. was supported by the grant GA\v CR P201/12/2613. The research of P.M was supported by the grant UMO-2012/07/B/ST1/03417.

{
\renewcommand{\refname}{\small References}
\bibliographystyle{symposium}

}

\end{document}